\documentclass[leqno]{prims}
\usepackage{hyperref}
\usepackage{amssymb,amsmath}

\usepackage{mathtools}
\usepackage[dvipdfm]{graphicx}
\usepackage{mathrsfs}

\newtheorem{thm}{Theorem}[section]
\newtheorem{cor}[thm]{Corollary}
\newtheorem{lem}[thm]{Lemma}

\newtheorem{prop}[thm]{Proposition}

\theoremstyle{definition}

\newtheorem{example}[thm]{Example}
\newtheorem{remark}[thm]{Remark}

\numberwithin{equation}{section}  

\makeatletter
\newenvironment{proof*}[1][\proofname]{\vskip-\lastskip\par
\vskip 8pt plus2pt minus2pt
  \normalfont
\noindent  {\itshape
  #1\@addpunct{.}}\hspace{.55em}\mdseries \ignorespaces
}{%
  {\vskip 8pt plus2pt minus2pt }
}
\makeatother
\newcommand{\Thm}[1]{Theorem~\ref{th:#1}}
\newcommand{\Prop}[1]{Proposition~\ref{prop:#1}}
\newcommand{\Lem}[1]{Lemma~\ref{lem:#1}}
\newcommand{\Cor}[1]{Corollary~\ref{cor:#1}}

\newcommand{\Rem}[1]{Remark~\ref{rem:#1}}
\newcommand{\Exam}[1]{Example~\ref{ex:#1}}
\newcommand{\Fig}[1]{\ref{fig:#1}}
\newcommand{\Eq}[1]{\eqref{eq:#1}}
\makeatletter
\def\rom#1{\mbox{\leavevmode\skip@\lastskip\unskip\/\ifdim\skip@=\z@\else\hskip\skip@\fi{\rm{#1}}}}
\makeatother
\renewcommand{\labelenumi}{{\rom{(\roman{enumi})}}}
\renewcommand{\a}{\alpha}\renewcommand{\b}{\beta}
\newcommand{\gm}{\gamma}\newcommand{\dl}{\delta}
\newcommand{\eps}{\varepsilon}
\newcommand{\lm}{\lambda}
\newcommand{\sg}{\sigma}
\newcommand{\ph}{\varphi}
\newcommand{\om}{\omega}
\newcommand{\Gm}{\Gamma}
\newcommand{\Lm}{\Lambda}
\newcommand{\Sg}{\Sigma}
\newcommand{\bfb}{\boldsymbol{b}}
\newcommand{\bfh}{{\boldsymbol{h}}}
\newcommand{\bfr}{\boldsymbol{r}}
\newcommand{\bfone}{\boldsymbol{1}}
\newcommand{\C}{\mathbb{C}}
\newcommand{\N}{\mathbb{N}}
\newcommand{\R}{\mathbb{R}}
\newcommand{\Z}{\mathbb{Z}}
\newcommand{\cD}{\mathcal{D}}
\newcommand{\cE}{\mathcal{E}}
\newcommand{\cF}{\mathcal{F}}
\newcommand{\cI}{\mathcal{I}}
\newcommand{\cH}{\mathcal{H}}
\newcommand{\cL}{\mathcal{L}}
\newcommand{\cO}{\mathcal{O}}
\newcommand{\cP}{\mathcal{P}}
\newcommand{\cQ}{\mathcal{Q}}
\newcommand{\sd}{\mathsf d}
\newcommand{\sJ}{\mathscr{J}}
\newcommand{\la}{\langle}\newcommand{\ra}{\rangle}
\newcommand{\wg}{\wedge}
\newcommand{\tr}{\mathop{\mathrm{tr}}\nolimits}
\newcommand{\diam}{\mathop{\mathrm{diam}}\nolimits}
\newcommand{\conn}[1]{\xleftrightarrow[#1]{}}

\VolumeNo{4x}           
\YearNo{201x}          
\communication{T. Kumagai. Received November 7, 2012. Revised August 6, 2013, October 11, 2013.} 

\allowdisplaybreaks[3]
\begin{document}


\TitleHead{Geodesic Distances and Intrinsic Distances}
\title{Geodesic Distances and Intrinsic Distances\\ on Some Fractal Sets}

\dedicatory{Dedicated to Professor Ichiro Shigekawa on the occasion of his 60th birthday}   

\AuthorHead{M. Hino}
\author{Masanori \textsc{Hino}\footnote{M. Hino: 
Graduate School of Engineering Science,
Osaka University,
Osaka 560-8531, Japan;
\email{hino@sigmath.es.osaka-u.ac.jp}}}


\maketitle

\begin{abstract}
Given strong local Dirichlet forms and $\R^N$-valued functions on a metrizable space, we introduce the concepts of geodesic distance and intrinsic distance on the basis of these objects.
They are defined in a geometric and an analytic way, respectively, and they are closely related with each other in some classical situations.
In this paper, we study the relations of these distances when the underlying space has a fractal structure.
In particular, we prove their coincidence for a class of self-similar fractals.
\end{abstract}

\section{Introduction}
For the analysis of strong local Dirichlet forms $(\cE,\cF)$ on a metric measure space $(K,\mu)$, the intrinsic distance 
defined as
\[
\sd(x,y)=\sup\{f(y)-f(x)\mid f\in\cF_{\mathrm{loc}}\cap C(K)\text{ and }\mu_{\la f\ra}\le \mu\},
 \quad x,y\in K,
\]
often plays a crucial role.
Here, $\cF_{\mathrm{loc}}$ represents the space of functions locally in $\cF$ and $\mu_{\la f\ra}$ denotes the energy measure of $f$.
For example, in a general framework, the off-diagonal Gaussian estimate and the Varadhan estimate of the transition density associated with $(\cE,\cF)$ are described on the basis of the intrinsic distance (see, e.g., \cite{St95,No97,Da} and the references therein).
When the underlying space has a Riemannian structure, the geodesic distance $\rho(x,y)$ is also defined as the infimum of the length of continuous curves connecting $x$ and $y$, and $\sd$ and $\rho$ coincide with each other under suitable conditions.

In this paper, we focus on the case when $K$ does not have a differential structure, in particular, when $K$ is a fractal set, and we study the relation between two distances that are defined in a way similar to $\sd$ and $\rho$.
The straightforward formulation of this problem, however, does not work well.
This is because 
in typical examples such as the canonical Dirichlet forms on Sierpinski gaskets with the Hausdorff measure~$\mu$, the energy measures are always singular to $\mu$ (see, e.g., \cite{Hi05,HN06,Ku89}); accordingly, $\sd$ vanishes everywhere.
This is closely related to the fact that the transition density exhibits sub-Gaussian behavior.
Nevertheless, if the reference measure in the definition of $\sd$ is replaced suitably, we can obtain a nontrivial intrinsic distance.
Indeed, Kigami~\cite{Ki08} and Kajino~\cite{Ka12} studied, following Metz and Sturm~\cite{MS95}, the canonical Dirichlet form on the 2-dimensional standard Sierpinski gasket with the underlying measure $\mu_{\la h_1\ra}+\mu_{\la h_2\ra}$, where the pair $h_1$ and $h_2$ is taken as the orthonormal system of the space of harmonic functions.
In such a case, the mapping $\bfh:=(h_1,h_2)\colon K\to\R^2$ provides a homeomorphism of $K$ to its image (\cite{Ki93}). In particular, they proved that
\begin{itemize}
\item the intrinsic distance $\sd_\bfh$ on $K$ coincides with the geodesic distance $\rho_\bfh$ on $\bfh(K)$ by identifying $K$ and $\bfh(K)$;
\item the transition density associated with $(\cE,\cF)$ on $L^2(K,\mu)$ has off-diagonal Gaussian estimates by using such distances.
\end{itemize}
In this paper, we study the relation between $\sd_\bfh$ and $\rho_\bfh$ (defined on $K$) in more general frameworks.
First, we prove the one-sided inequality $\rho_\bfh\le \sd_\bfh$ when the underlying spaces have finitely ramified cell structures (\Thm{1}).
The reverse inequality is proved under tighter constraints on self-similar Dirichlet forms on a class of self-similar fractals (\Thm{2}); typical examples are the standard Dirichlet forms on the 2-dimensional generalized Sierpinski gaskets.
Both the proofs are based on purely analytic arguments, unlike the corresponding proof in \cite{Ka12}, where detailed information of the transition density was utilized, together with probabilistic arguments.
Our results are applicable to some examples in which the precise behaviors of the associated transition densities are not known.
The crucial part of the proof of \Thm{1} is that the truncated geodesic distance function based on $\bfh$ satisfies the conditions in the definition of $\sd_\bfh$.
To prove this claim, we show that a discrete version of the geodesic distance has some good estimates and that the limit function inherits them.
The proof of \Thm{2} is more tricky.
The key lemma (\Lem{2.3}) is an analog of the classical fact on domain $D$ of $\R^d$, stating that any function $f\in W^{1,1}(D)$ with $|\nabla f|_{\R^d}\le 1$ a.e.\ is locally Lipschitz with a local Lipschitz constant less than or equal to $1$.
We prove that $\sd_\bfh(x,y)\le (1+\eps)\rho_\bfh(x,y)$ if $x$ and $y$ are suitably located.
An inequality of this type is not evident in the nonsmooth setting;
the hidden obstacle is that a type of ``Riemannian metric'' which $K$ is equipped with is degenerate almost everywhere (cf.~\cite{Hi10,Hi12,Ku89}), and we have \textit{a priori} the inequality stated above only for the points that are nondegenerate with respect to $\bfh$.
Using a rather strong assumption ((B1) in Section~2), we can take sufficiently many such good points on arbitrary continuous curves, which enables us to deduce the inequality $\sd_\bfh\le\rho_\bfh$.
At the moment, we need various assumptions to obtain such estimates owing to the lack of more effective tools for analysis.
However, we expect the claims of theorems in this paper to be valid in much more general situations, possibly with an appropriate modification of the framework (see also \Rem{classical} for further discussion).

The remainder of this article is organized as follows. 
In Section~2, we present the framework and state the main theorems.
In Sections~3 and 4, we prove Theorems~\ref{th:1} and \ref{th:2}, respectively.
\section{Framework and results}
Let $K$ be a compact metrizable space, and $\mu$, a finite Borel measure on $K$ with full support.
Let $d_K$ denote a metric on $K$ that is compatible with the topology.
For subsets $U$ of $K$, we denote the closure, interior, and boundary of $U$ by $\overline{U}$, $U^\circ$ and $\partial U$, respectively.
The set of all real-valued continuous functions on $K$ is represented as $C(K)$, which is equipped with the uniform topology.

Let $(\cE,\cF)$ be a strong local regular (symmetric) Dirichlet form on $L^2(K,\mu)$.
For simplicity, we write $\cE(f)$ for $\cE(f,f)$.
The space $\cF$ is regarded as a Hilbert space with the inner product $(f,g)_{\cF}:=\cE(f,g)+\int_{K}fg\,d\mu$. 
For $f\in\cF$, $\mu_{\la f\ra}$ denotes the energy measure of $f$,
that is, when $f$ is bounded, $\mu_{\la f\ra}$ is characterized by the identity
\[
\int_K \ph\,d\mu_{\la f\ra}=2\cE(f,f\ph)-\cE(f^2,\ph)
\quad\mbox{for all }\ph\!\in\!\cF\cap C(K);
\]
for general $f\in\cF$, $\mu_{\la f\ra}$ is defined by the natural approximation.
Let $N\in\N$ and $\bfh=(h_1,\dots,h_N)$ such that $h_j\in\cF\cap C(K)$ for every $j=1,\dots,N$.
Let $\mu_{\la\bfh\ra}$ denote $\sum_{j=1}^N \mu_{\la h_j\ra}$.
Then, the {\em intrinsic distance} based on $(\cE,\cF)$ and $\mu_{\la\bfh\ra}$ is defined as
\begin{equation}\label{eq:sd}
 \sd_\bfh(x,y):=\sup\{f(y)-f(x)\mid f\in\cF\cap C(K)\text{ and }\mu_{\la f\ra}\le \mu_{\la \bfh \ra}\},
 \quad x,y\in K.
\end{equation}
We remark that the underlying measure $\mu$ does not play an essential role in \Eq{sd}. 
Further, we do not assume the absolute continuity of energy measures with respect to $\mu$ or $\mu_{\la \bfh\ra}$.
For a continuous curve $\gamma\in C([s,t]\to K)$, its length based on $\bfh$ is defined as
\begin{align*}
 \ell_\bfh(\gamma):=\sup\biggl\{&\sum_{i=1}^n|\bfh(\gamma(t_{i}))-\bfh(\gamma(t_{i-1}))|_{\R^N}\biggm|
 n\in\N,\ s=t_0<t_1<\dots<t_n=t\biggr\},
\end{align*}
where $|\cdot|_{\R^N}$ denotes the Euclidean norm on $\R^N$.
This is nothing but the pullback of the concept of the usual length of curves in $\R^N$ by the map $\bfh$.
Then, the {\em geodesic distance} based on $\bfh$ is defined as
\[
  \rho_\bfh(x,y):=\inf\{\ell_\bfh(\gamma)\mid \gamma\in C([0,1]\to K),\ \gamma(0)=x,\ \mbox{and }\gamma(1)=y\},
\quad x,y\in K,
\]
where $\inf\emptyset:=\infty$.
If $\gm\in C([s,t]\to K)$ satisfies that $\gm(s)=x$, $\gm(t)=y$, and $\rho_\bfh(x,y)=\ell_\bfh(\gm)$, we say that $\gm$ is a shortest path connecting $x$ and $y$.

We note that the two distances introduced here can be defined for more general underlying spaces such as locally compact spaces, by slight modifications if necessary.
In this paper, however, we consider only compact spaces for simplicity.
\begin{remark}
We have the following properties.
\begin{enumerate}
\item Both $\sd_\bfh$ and $\rho_\bfh$ are ($[0,+\infty]$-valued) quasi-metrics on $K$, that is, the distance between two distinct points may be zero, but all the other axioms of metric are satisfied
(see \Cor{metric} for further discussion).
\item Let $\gm\in C([s,t]\to K)$. If $\{s_n\}$ decreases to $s$ and $\{t_n\}$ increases to $t$, then $\lim_{n\to\infty}\ell_\bfh(\gm|_{[s_n,t_n]})=\ell_\bfh(\gm)$.
\item If $\gm\in C([0,1]\to K)$ is a shortest path connecting $x$ and $y$ with $\rho_\bfh(x,y)<\infty$, then for any $0\le s<t\le1$, $\gm|_{[s,t]}$ is a shortest path connecting $\gm(s)$ and $\gm(t)$.
\item  If $\bfh\colon K\to\R^N$ is injective, then for any $x,y\in K$, $\rho_\bfh(x,y)$ coincides with the geodesic distance between $\bfh(x)$ and $\bfh(y)$ in $\bfh(K)\subset \R^N$ on the basis of the Euclidean distance.
\end{enumerate}
\end{remark}
In order to state the first theorem, we consider the following conditions.
\begin{enumerate}\def\labelenumi{(A\arabic{enumi})}
\item There exists an increasing sequence of nonempty finite subsets $\{V_m\}_{m=0}^\infty$ of $K$ such that the following hold:
\begin{enumerate}\def\labelenumii{(\roman{enumii})}
\item For each $m$, $K\setminus V_m$ is decomposed into finitely many connected components $\{U_\lm\}_{\lm\in\Lambda_m}$;
\item For every $x\in K$, the sets $\{\bigcup_{\lm\in \Lm_m;\ x \in \overline{U_\lm}}\overline{U_\lm} \}_{m=0}^\infty$ constitute a fundamental system of neighborhoods of $x$.
\end{enumerate}
\item $\cF\subset C(K)$.
\item $\cE(f,f)=0$ if and only if $f$ is a constant function.
\end{enumerate}
We give several remarks.
In \Lem{connectedness} below, it is proved from conditions~(A1)--(A3) that $K$ is arcwise connected.
Then, it is easy to prove that $V_*:=\bigcup_{m=0}^\infty V_m$ is dense in $K$.
From the closed graph theorem, (A2) implies that $\cF$ is continuously imbedded in $C(K)$.
Condition~(A3) is equivalent to the irreducibility of $(\cE,\cF)$ in this framework, from \cite[Theorem~2.1.11]{CF}, for example.
For $\lm\in\Lm_m$ and $\lm'\in\Lm_{m'}$ with $m\le m'$, either $U_\lm\supset U_{\lm'}$ or $U_\lm\cap U_{\lm'}=\emptyset$ holds.
A slightly different version of (A1) was discussed in \cite{Te08} and named finitely ramified cell structure.

The first main theorem is stated as follows.
\begin{thm}\label{th:1}
Suppose \rom{(A1)}--\rom{(A3)}.
Then, $\rho_\bfh(x,y)\le \sd_\bfh(x,y)$ for all $x,y\in K$.
\end{thm}
To obtain the reverse inequality, we need tighter constraints.
Following\break Kigami~\cite{Ki}, we introduce the concepts of post-critically finite self-similar sets and harmonic structures associated with them.
Let $\Z_+$ denote the set of all nonnegative integers.
Let $S$ be a finite set with $\#S\ge2$.
For $i\in S$, let $\psi_i\colon K\to K$ be a continuous injective map.
Set $\Sg=S^\N$. For $i\in S$, we define a shift operator $\sg_i\colon \Sg\to\Sg$ as $\sg_i(\om_1\om_2\cdots)=i\om_1\om_2\cdots$.
We assume that there exists a continuous surjective map $\pi\colon \Sg\to K$ such that $\psi_i\circ \pi=\pi\circ\sg_i$ for every $i\in S$.
Then, $\cL=(K,S,\{\psi_i\}_{i\in S})$ is called a self-similar structure.

We set $W_0=\{\emptyset\}$ and $W_m=S^m$ for $m\in \N$.
For $w=w_1w_2\cdots w_m\in W_m$, let $\psi_w$ denote $\psi_{w_1}\circ\psi_{w_2}\circ\cdots\circ\psi_{w_m}$ and let $K_w$ denote $\psi_w(K)$.
By convention, $\psi_\emptyset$ is the identity map from $K$ to $K$.
Let 
\[
\cP=\bigcup_{m=1}^\infty \sg^m\left(\pi^{-1}\left(\bigcup_{i,j\in S,\,i\ne j}(K_i\cap K_j)\right)\right)\quad\text{and}\quad V_0=\pi(\cP),
\]
where $\sg^m\colon\Sigma\to\Sigma$ is defined as $\sg^m(\om_1\om_2\cdots)=\om_{m+1}\om_{m+2}\cdots$.
The set $\cP$ is called the post-critical set.
We assume that $K$ is connected and that the self-similar structure $\cL$ is {post-critically finite} (p.c.f.), that is, $\cP$ is a finite set.
For $m\in\N$, let $V_m=\bigcup_{w\in W_m}\psi_w(V_0)$.

For a finite set $V$, $l(V)$ denotes the space of all real-valued functions on $V$.
We equip $l(V)$ with an inner product $(\cdot,\cdot)_{l(V)}$   defined by $(u,v)_{l(V)}=\sum_{q\in V}u(q)v(q)$.
The norm derived from $(\cdot,\cdot)_{l(V)}$ is denoted by $|\cdot|_{l(V)}$.
Let $D=(D_{qq'})_{q,q'\in V_0}$ be a symmetric linear operator on $l(V_0)$ (also considered to be a square matrix of size $\#V_0$) such that the following conditions hold:
\begin{enumerate}
\item[(D1)] $D$ is nonpositive definite;
\item[(D2)] $Du=0$ if and only if $u$ is constant on $V_0$;
\item[(D3)] $D_{qq'}\ge0$ for all $q\ne q'\in V_0$.
\end{enumerate}
We define $\cE^{(0)}(u,v)=(-Du,v)_{l(V_0)}$ for $u,v\in l(V_0)$.
This is a Dirichlet form on the $L^2$ space on $V_0$ with the counting measure (cf.~\cite[Proposition~2.1.3]{Ki}).
For $\bfr=\{r_i\}_{i\in S}$ with $r_i>0$, we define a bilinear form $\cE^{(m)}$ on $l(V_m)$ as
\begin{equation*}
  \cE^{(m)}(u,v)=\sum_{w\in W_m}\frac{1}{r_w}\cE^{(0)}(u\circ\psi_w|_{V_0},v\circ\psi_w|_{V_0}),\quad
  u,v\in l(V_m).
\end{equation*}
Here, $r_w=r_{w_1}r_{w_2}\cdots r_{w_m}$ for $w=w_1w_2\cdots w_m\in W_m$ and $r_\emptyset=1$.
We call $(D,\bfr)$ a {regular harmonic structure} if $0<r_i<1$ for all $i\in S$ and
\[
\cE^{(0)}(v,v)=\inf\{\cE^{(1)}(u,u)\mid u\in l(V_1)\mbox{ and }u|_{V_0}=v\}
\]
for every $v\in l(V_0)$.
Then, $\cE^{(m)}(u|_{V_m},u|_{V_m})\le \cE^{(m+1)}(u,u)$ for $m\in\Z_+$ and $u\in l(V_{m+1})$.
The existence of harmonic structures is a nontrivial problem.
It is known that all nested fractals have canonical regular harmonic structures (\cite{Li}; see also \cite{Ki}).

We assume that a regular harmonic structure $(D,\bfr)$ is given.
Let $\mu$ be a finite Borel measure on $K$ with full support.
We can then define a strong local and regular Dirichlet form $(\cE,\cF)$ on $L^2(K,\mu)$ associated with $(D,\bfr)$ by
\begin{align*}
\cF&=\left\{u\in C(K)\subset L^2(K,\mu)\left|\,
\lim_{m\to\infty}\cE^{(m)}(u|_{V_m},u|_{V_m})<\infty\right.\right\},\\
\cE(u,v)&= \lim_{m\to\infty}\cE^{(m)}(u|_{V_m},v|_{V_m}),\quad u,v\in\cF
\end{align*}
(see the beginning of \cite[Section~3.4]{Ki}).
Then, conditions (A1)--(A3) are satisfied. 
((A1) is guaranteed by \cite[Proposition~1.6.8~(2) and Proposition~1.3.6]{Ki}.)

For a map $\psi\colon K\to K$ and a function $f$ on $K$, $\psi^* f$ denotes the pullback of $f$ by $\psi$, that is, $\psi^* f=f\circ \psi$.
The Dirichlet form $(\cE,\cF)$ constructed above satisfies the self-similarity
\[
  \cE(f,g)=\sum_{i\in S}\frac1{r_i}\cE(\psi_i^* f,\psi_i^* g),\quad f,g\in\cF.
 \]
For each $u\in l(V_0)$, there exists a unique function $h\in\cF$ such that $h|_{V_0}=u$ and $\cE(h)=\inf\{\cE(g)\mid g\in\cF,\ g|_{V_0}=u\}$.
Such a function $h$ is termed a {harmonic function}.
The space of all harmonic functions is denoted by $\cH$.
For any $w\in W_*$ and $h\in\cH$, $\psi_w^* h\in\cH$.
We can identify $\cH$ with $l(V_0)$ by the linear map $\iota\colon l(V_0)\ni u\mapsto h\in \cH$.
In particular, $\cH$ is a finite dimensional subspace of $\cF$.
For each $i\in S$, we define a linear operator $A_i\colon l(V_0)\to l(V_0)$ as 
$A_i=\iota^{-1}\circ \psi_i^* \circ \iota$,
which is also considered as a square matrix of size $\#V_0$.
For $i\ne j\in S$, the fixed points $p_i$ and $p_j$ of $\psi_i$ and $\psi_j$, respectively, are different by \cite[Lemma~1.3.14]{Ki}.
We set 
\[
S_0=\{i\in S\mid \mbox{the fixed point $p_i$ of $\psi_i$ belongs to $V_0$}\}.
\]
For $i\in S_0$, $r_i$ is an eigenvalue of $A_i$, and we can take its eigenvector $v_i\in l(V_0)$ whose components are all nonnegative (cf.~\cite[Theorem~A.1.2]{Ki}).
Note that $v_i(p_i)=0$ since $r_i\ne1$.

We now consider the following conditions:
\begin{enumerate}\def\labelenumi{(B\arabic{enumi})}
\item $\#V_0=3$;
\item For all $p\in V_0$, $K\setminus \{p\}$ is connected;
\item $\#S_0=3$, that is, each $p\in V_0$ is the fixed point of $\psi_i$ for some $i\in S_0$. Moreover, $Dv_i(q)<0$ for every $q\in V_0\setminus \{p\}$;
\item For every $i\in S_0$, $A_i$ is invertible.
\end{enumerate}
We remark that, in condition (B3), $v_i(q)>0$ follows in addition for every $q\in V_0\setminus\{p\}$ from (B2) and \cite[Corollary~A.1.3]{Ki}.
\begin{thm}\label{th:2}
Suppose \rom{(B1)}--\rom{(B4)}.
Take $\bfh=(h_1,\dots,h_N)$ such that each $h_j$ is a harmonic function.
Then, $\sd_\bfh(x,y)= \rho_\bfh(x,y)$ for all $x,y\in K$.
\end{thm}

Typical examples that meet conditions~(B1)--(B4) are given below.
\begin{example}\label{ex:gasket}
\begin{figure}\centering
\includegraphics{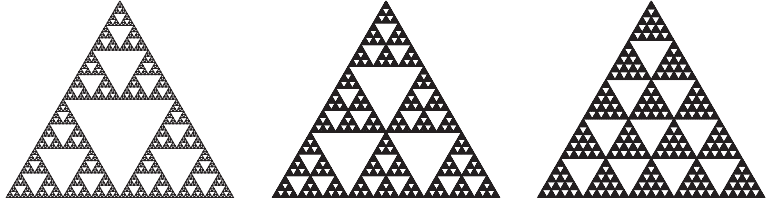}
\caption{2-dimensional level $l$ Sierpinski gaskets ($l=2,3,5$)}
\label{fig:1}
\end{figure}
Take the 2-dimensional level $l$ Sierpinski gasket as $K$ (see Figure~\Fig{1}).
The set $V_0$ consists of the three vertices $p_1$, $p_2$, and $p_3$ of the largest triangle in $K$.
For $i=1,2,3$, let $\psi_i$ denote the map whose fixed point is $p_i$ among the contraction maps constructing $K$.
Since $K$ is a nested fractal, there exists a canonical regular harmonic structure $(D,\bfr)$ corresponding to the Brownian motion on $K$.
The matrix $D$ is given by $D=(D_{p_ip_j})_{i,j=1}^3=\begin{bmatrix}-2&1&1\\1&-2&1\\1&1&-2\end{bmatrix}$.
The eigenvector $v_1$ is described as $v_1=\begin{bmatrix}0\\1\\1\end{bmatrix}$ by symmetry; thus, $Dv_1=\begin{bmatrix}2\\-1\\-1\end{bmatrix}$. Similarly, the vectors $v_i$ and $Dv_i$ for $i=2,3$ are described as
\[
v_2=\begin{bmatrix}1\\0\\1\end{bmatrix},~
Dv_2=\begin{bmatrix}-1\\2\\-1\end{bmatrix},~
v_3=\begin{bmatrix}1\\1\\0\end{bmatrix},~
Dv_3=\begin{bmatrix}-1\\-1\\2\end{bmatrix}.
\]
Therefore, conditions~(B1)--(B3) hold.
Condition~(B4) is also verified directly.
We note that the detailed information of the transition density associated with $(\cE,\cF)$ on $L^2(K,\mu_{\la \bfh \ra})$ is known only for $l=2$ (see \cite{Ki08,Ka12}), since we cannot expect the volume doubling property of $\mu_{\la \bfh \ra}$ if $l\ge3$.
\end{example}
The following examples are based on the suggestion of the referee.
\begin{example}
\begin{figure}\centering
\includegraphics{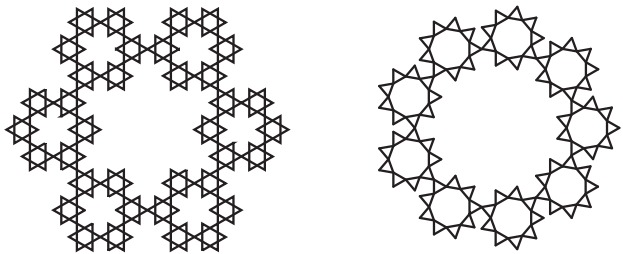}
\caption{Hexagasket $(n=6)$ and Nonagasket $(n=9)$}
\label{fig:2}
\end{figure}
Let $n$ be 6 or 9.
Let $S=\{0,1,\dots,n-1\}$ and $p_k=\exp(2\pi k\sqrt{-1}/n)$ for $k\in \Z$.
For $k\in S$, we define $\psi_k\colon \C \to\C$ by 
\[
 \psi_k(z)=p_k\{\b_n(z-1)+1\},
\]
where $\b_n=2/(3+\sqrt3\cot (\pi/n))$.
Let $K$ be the unique nonempty compact subset of $\C$ such that $K=\bigcup_{k\in S}\psi_k(K)$ (see Figure~\Fig{2} and also \cite[Example~7.4]{Te00}). Then, the triplet $(K,S,\{\psi_k|_K\}_{k\in S})$ constitutes a self-similar structure, $\#\cP=3$, and $V_0=\{p_0,p_{n/3},p_{2n/3}\}$. Note that $\b_n$ is taken so that $\psi_0(K)\cap\psi_1(K)$ is a one-point set, and that $\#V_0$ is not $n$ but 3 since $\psi_j$ involves a rotation.
We can construct a canonical harmonic structure as in \Exam{gasket} such that conditions~(B1)--(B4) hold.
\end{example}
\begin{remark}\label{rem:classical}
 Let us consider the classical case for comparison.
Let $K$ be a nonempty compact set of $\R^m$ such that $\overline{K^\circ}=K$ and $\partial K$ is a smooth hypersurface.
Let $(\cE,\cF)$ be a Dirichlet form on $L^2(K,dx)$ that is given by
\[
  \cE(f,g)=\frac12\int_K \sum_{i,j=1}^m a_{ij}(x)\frac{\partial f}{\partial x_i}(x)\frac{\partial g}{\partial x_j}(x)\,dx,
  \quad f,g\in\cF:= H^1(K).
\]
Here, $A(x)=(a_{ij}(x))_{i,j=1}^m$ is a symmetric, bounded, and uniformly positive definite matrix-valued continuous function on $K$.
Let $N\ge m$ and let $h_i$ be a Lipschitz function on $K$ for  $i=1,\dots,N$.
Let $B$ be an $N\times m$-matrix valued function on $K$ such that the $i$th row of $B(x)$ is equal to $^t\nabla h_i(x)$ for  $i=1,\dots,N$ and $x\in K$.
Assume that $A(x)$ is connected with $(h_1,\dots, h_N)$ by the identity $^t B(x)B(x)=A(x)^{-1}$ for a.e.~$x$.
Then, we have
\begin{align*}
d\mu_{\la \bfh\ra}
&=\sum_{i=1}^N(A(x)\nabla h_i(x),\nabla h_i(x))_{\R^m}\,dx\\
&=\sum_{i=1}^N {}^t\nabla h_i(x)A(x)\nabla h_i(x)\,dx\\
&=\tr \bigl(B(x)A(x)\,^t\! B(x)\bigr)\,dx\\
&=\tr \bigl(A(x)\,^t\! B(x)B(x)\bigr)\,dx\\
&= m\cdot dx.
\end{align*}
Therefore, $\sd_\bfh(x,y)$ should be defined as
\begin{align*}
&\sd_\bfh(x,y)\\
&=\sup\{f(y)-f(x)\mid f\in H^1(K)\cap C(K)\mbox{ and } (A(z)\nabla f(z),\nabla f(z))_{\R^m}\le m \mbox{ a.e.}\,z\}.
\end{align*}
Moreover, for $\bfb\in\R^m$ and $x\in K$,
\begin{align*}
\sum_{i=1}^N (\nabla h_i(x),\bfb)_{\R^m}^2
&=|B(x)\bfb|_{\R^N}^2
=\,^t\bfb\,^t\!B(x)B(x)\bfb
=|A(x)^{-1/2}\bfb|_{\R^m}^2,
\end{align*}
which implies that for $\gm\in C([0,1]\to K)$ that is piecewise smooth,
\begin{align*}
\ell_\bfh(\gm)
&=\int_0^1 \left|\frac{d}{dt}(\bfh\circ \gm)(t)\right|_{\R^N}dt\\
&=\int_0^1 \left(\sum_{i=1}^N\bigl((\nabla h_i)(\gm(t)),\dot\gm(t)\bigr)_{\R^m}^2\right)^{1/2}\,dt\\
&=\int_0^1 \left| A(\gm(t))^{-1/2}\dot \gm(t)\right|_{\R^m} dt.
\end{align*}
Thus, $\sd_\bfh(x,y)=\sqrt m \rho_\bfh(x,y)$ holds.\footnote{Though this type of identity ought to be a known result (cf.\ \cite{DP91}), we give a sketch of the proof. For $x\in K$ and $M>0$, the function $f:=\rho_\bfh(\cdot,y)\wg M$ on $K$ is locally Lipschitz and $|(\nabla f(x),\bfb)_{\R^m}|\le|A(x)^{-1/2}\bfb|_{\R^m}$ for any $\bfb\in\R^m$.
By letting $\bfb=A(x)\nabla f(x)$, we obtain $|A(x)^{1/2}\nabla f(x)|_{\R^m}\le1$. This implies that $\rho_\bfh(x,y)\wg M\le \sd_\bfh(x,y)/\sqrt m$ for $x\in K$.
On the other hand, for $f\in C^1(K)$, \[f(y)-f(x)=\int_0^1 (\nabla f(\gm(t)),\dot\gm(t))_{\R^m}\,dt\le \int_0^1|A(\gm(t))^{1/2}\nabla f(\gm(t))|_{\R^m}|A(\gm(t))^{-1/2}\dot \gm(t)|_{\R^m} dt.\]
Let $\gm\in C([0,1]\to K)$ be a piecewise-linear curve connecting $x$ and $y$, and $f$, a function on $K$ satisfying the condition in the definition of $\sd_\bfh$.
 Then the inequality 
 \[
 f(y)-f(x)\le \sqrt m\int_0^1|A(\gm(t))^{-1/2}\dot \gm(t)|_{\R^m} dt=\sqrt m\ell_\bfh(\gm)
 \]
 holds, by approximating $f$ by smooth functions if necessary. Taking supremum and infimum with respect to $f$ and $\gm$, respectively, we obtain that $\sd_\bfh(x,y)\le\sqrt m \rho_\bfh(x,y)$.}
This example shows that the information of the dimension of $K$ is required to identify $\sd_\bfh$ with $\rho_\bfh$ in general.
The author guesses that the correct measure to define the intrinsic metric is $p(x)^{-1}\,d\mu_{\la \bfh\ra}$ instead of $\mu_{\la \bfh\ra}$, where $p(x)$ is the pointwise index defined in \cite{Hi10} and represents the effective dimension of a type of tangent space at $x$ (see also \cite{Hi12}); the treatment of such a measure is beyond the scope of this paper.
Under the assumptions of \Thm{2}, $p(x)=1$ for $\mu_{\la \bfh\ra}$-a.e.\ $x$ from the result of \cite{Hi08}. 
Therefore, this guess is also consistent with \Thm{2}.
Such examples show that the problem discussed in this paper is more intricate than it seems.
\end{remark}

\section{Proof of \Thm{1}}
First, we remark on some properties of energy measures associated with strong local Dirichlet forms.
For the proof, see \cite[Theorem~4.3.8]{CF}, for example.
In the statement of \Lem{en}, $\tilde f$ denotes the quasi-continuous modification of $f$.
\begin{lem}\label{lem:en}
For every $f\in\cF$, the push-forward measure of $\mu_{\la f\ra}$ by $\tilde f$ is absolutely continuous with respect to the 1-dimensional Lebesgue measure.
In particular, we have the following.
\begin{enumerate}
\item For any $f\in\cF$, $\mu_{\la f\ra}$ does not have a point mass.
\item For $f,g\in\cF$, $\mu_{\la f\ra}=\mu_{\la g\ra}$ on the set $\{\tilde f=\tilde g\}$.
\end{enumerate}
\end{lem}
In this paper, we consider only the case $\cF\subset C(K)$; accordingly, any $f\in\cF$ is continuous from the beginning.

We also remark that 
\begin{equation}\label{eq:cE}
\cE(f)=\frac12\mu_{\la f\ra}(K)
\quad\mbox{for any $f\in\cF$}
\end{equation}
(see \cite[Corollary~3.2.1]{FOT}).

For $f,g\in\cF$, a signed measure $\mu_{\la f,g\ra}$ on $K$ is defined as $\mu_{\la f,g\ra}=(\mu_{\la f+g\ra}-\mu_{\la f\ra}-\mu_{\la g\ra})/2$.
It is bilinear in $f,g$ and $\mu_{\la f,f\ra}=\mu_{\la f\ra}$ holds.

In the remainder of this section, we always assume (A1)--(A3).
We state some basic properties in the following series of lemmas.
\begin{lem}\label{lem:1}
Let $m\in\Z_+$ and $\lm\in\Lm_m$.
Then, $U_\lm$ is open, $\overline{U_\lm}\subset U_\lm\cup V_m$, and $\partial U_\lm\subset V_m$.
\end{lem}
\begin{proof}
If $\overline{U_\lm}\cap U_{\kappa}\ne\emptyset$ for some $\kappa\in \Lm_m\setminus\{\lm\}$, then $U_\lm\cup U_{\kappa}$ is connected (see, e.g., \cite[Theorem~6.1.9]{En}), which is a contradiction.
Thus, $\overline{U_\lm}\subset K\setminus\bigcup_{\kappa\in \Lm_m\setminus\{\lm\}}U_{\kappa}= U_\lm\cup V_m$.
Similarly, we have $K\setminus U_\lm=V_m\cup \bigcup_{\kappa\in \Lm_m\setminus\{\lm\}}\overline{U_{\kappa}}$, which is a closed set.
Since $\partial U_\lm=\overline{U_\lm}\setminus{U_\lm}\subset V_m$, the last claim follows.
\end{proof}
\begin{lem}\label{lem:connectedness}
$K$ is arcwise connected.
\end{lem}
\begin{proof}
First, we prove that $K$ is connected. If $K$ is a disjoint union of nonempty open sets $K_1$ and $K_2$, then $\bfone_{K_1}\in \cF$ from the regularity of $(\cE,\cF)$. From the strong locality of $(\cE,\cF)$, $\cE(\bfone_{K_1})=0$, which contradicts (A3).
 
Then, $K$ is a compact, metrizable, connected, and locally connected space, which implies that $K$ is arcwise connected (see, e.g., \cite[Section~6.3.11]{En}).
\end{proof}
For a subset $U$ of $K$, let $\diam U$ denote the diameter of $U$ with respect to the metric $d_K$.
\begin{lem}\label{lem:diam}
As $m\to\infty$, $\sup\{\diam \overline{U_\lm}\mid \lm\in\Lm_m\}$ converges to $0$.
\end{lem}
\begin{proof}
Let $\eps>0$. From (A1)(ii), for each $x\in K$, there exists $m(x)\in \Z_+$ such that $\diam N_x\le \eps$, where $N_x:=\bigcup_{\lm\in \Lm_{m(x)};\, x\in\overline{U_\lm}}\overline{U_\lm}$.
Since $K$ is compact and covered with $\bigcup_{x\in K}N_x^\circ$, $K$ is described as $\bigcup_{j=1}^k N_{x_j}^\circ$ for some $\{x_j\}_{j=1}^k\subset K$. 
Let $M=\max\{m(x_j)\mid j=1,\dots,k\}$.
Then, for each $m\ge M$ and $\lm\in\Lm_m$, $U_\lm$ is a subset of some $N_{x_j}$, which implies that $\diam U_\lm\le \eps$. 
This indicates the claim.
\end{proof}

\begin{lem}\label{lem:lh}
For a continuous curve $\gamma\in C([s,t]\to K)$,
\begin{align*}
 &\ell_\bfh(\gamma)\\
 &=\sup\left\{\sum_{i=1}^n|\bfh(\gamma(t_{i}))-\bfh(\gamma(t_{i-1}))|_{\R^N}\biggm|
\begin{array}{l}n\in\N,\ s=t_0<t_1<\dots<t_n=t,\\\gm(t_i)\in V_*\text{ for every }i=1,\dots,n-1\end{array}\right\}.
\end{align*}
\end{lem}
\begin{proof}
This is evident from the fact that the set $\{u\in[s,t]\mid \gm(u)\in V_*\}$ is dense in $[s,t]$ if $\gm$ is not constant on any nonempty intervals.
\end{proof}

\begin{lem}\label{lem:Hm}
Let $V$ be a finite and nonempty subset of $K$ and $u\in l(V)$.
Then, there exists a unique function $g\in\cF$ such that $g$ attains the infimum of the set $\{\cE(f)\mid f\in \cF,\ f=u\text{ on }V\}$.
\end{lem}
\begin{proof}
The proof is standard.
From the regularity of $(\cE,\cF)$, for each $p\in V$, there exists $g\in \cF$ such that $g(p)=1$ and $g(q)=0$ for all $q\in V\setminus\{p\}$. Therefore, the set in the statement is nonempty.
Take functions $\{f_n\}$ from $\cF$ such that $f_n=u$ on $V$, $\min_{x\in V}u(x)\le f_n\le \max_{x\in V}u(x)$, and $\cE(f_n)$ converges to $\inf\{\cE(f)\mid f\in \cF,\ f=u\text{ on }V\}$.
Since $\{f_n\}$ is bounded in $\cF$, we can take a subsequence of $\{f_n\}$ such that its Ces\`aro means converge to some $g$ in $\cF$.
Then, $g$ attains the infimum.
If another $g'\in\cF$ attains the infimum, then $\cE(g-g')=2\cE(g)+2\cE(g')-4\cE((g+g')/2)\le0$. 
From (A3), $g-g'$ is a constant function. Since $g-g'=0$ on $V$, we conclude that $g=g'$, which ensures uniqueness of the minimizer.
\end{proof}
For $m\in\Z_+$ and $u\in l(V_m)$, let $H_m u$ denote the function $g$ in the above lemma with $V=V_m$.
For $m\in\Z_+$ and $f\in \cF$, let $H_m f$ denote $H_m(f|_{V_m})$ by abuse of notation.
The linearity of $H_m$ is a basic fact and the proof is omitted.
\begin{lem}\label{lem:zero}
Let $U$ be an open set of $K$, and let $f$ be a function in $\cF$ such that $f=0$ on $\partial U$.
Then, the function $f\cdot\bfone_{U}$ belongs to $\cF$.
\end{lem}
\begin{proof}
We may assume that $f$ is nonnegative.
For $\eps>0$, let $f_\eps=(f-\eps)\vee0$.
Then, $f_\eps=0$ on a certain open neighborhood $O_\eps$ of $\partial U$.
Take $\ph_\eps\in\cF$ such that $\ph_\eps=1$ on $U\setminus O_\eps$ and $\ph_\eps=0$ on $K\setminus \overline{U}$.
Then, $f_\eps\cdot\bfone_U=f_\eps\ph_\eps\in\cF$.
From \Lem{en} and \Eq{cE},
\[
\cE(f_\eps\cdot\bfone_U)=\frac12\mu_{\la f_\eps\cdot\bfone_U\ra}(U)
=\frac12\mu_{\la f_\eps\ra}(U)
\le \cE(f_\eps)\le \cE(f).
\]
Therefore, $\{f_{1/n}\cdot\bfone_U\}_{n\in\N}$ is bounded in $\cF$ and has a weakly convergent sequence in $\cF$. Since $f_\eps\cdot\bfone_U$ converges to $f\cdot\bfone_U$ pointwise as $\eps\downarrow0$, we conclude that $f\cdot\bfone_U\in \cF$.
\end{proof}
\begin{lem}\label{lem:min}
For $m\in\Z_+$, $\lm\in\Lm_m$, and $f\in\cF$, we have $\mu_{\la H_mf\ra}(U_\lm)\le \mu_{\la f\ra}(U_\lm)$.
\end{lem}
\begin{proof}
  Let $g=f\cdot\bfone_{U_\lm}+(H_m f)\cdot\bfone_{K\setminus U_\lm}$. 
  Since $g=H_mf+(f-H_mf)\cdot\bfone_{U_\lm}$, $g=f$ on $V_m$ and $g\in \cF$ by \Lem{zero}. 
  By combining the inequality $\mu_{\la H_m f\ra}(K)\le \mu_{\la g\ra}(K)$ and \Lem{en}, the claim holds.
\end{proof}
\begin{lem}\label{lem:expression}
Let $m\in \Z_+$ and $\lm\in \Lm_m$. Then, there exists a set $\{b_{pq}\}_{p,q\in \partial U_\lm}$ of real numbers such that $b_{pq}=b_{qp}\ge0$ for all $p\ne q$ and for every $f\in\cF$,
\[
\mu_{\la H_mf\ra}(U_\lm)=\frac12\sum_{p,q\in\partial U_\lm}b_{pq}(f(p)-f(q))^2.
\]
\end{lem}
\begin{proof}
From Lemmas~\ref{lem:zero} and \ref{lem:en}, $H_mf=0$ on $U_\lm$ if $f=0$ on $\partial U_\lm$. Therefore, if $f,f'\in\cF$ satisfy $f=f'$ on $\partial U_\lm$, then $\mu_{\la H_mf\ra}(U_\lm)=\mu_{\la H_mf'\ra}(U_\lm)$.
Thus, for $\ph,\psi\in l(\partial U_\lm)$,
\[
  \cQ(\ph,\psi):=\mu_{\la H_m f,H_m g\ra}(U_\lm),
  \quad 
\]
where $f,g\in\cF$ satisfy $f|_{\partial U_\lm}=\ph$ and $g|_{\partial U_\lm}=\psi$, is well-defined.
From the proof of \cite[Proposition~2.1.3]{Ki}, the claim of the lemma follows if we prove that $\cQ$ is a Dirichlet form on the $L^2$ space on $\partial U_\lm$ with respect to the counting measure.
The bilinearity and the nonnegativity of $\cQ$ are evident.
We prove the Markov property.
Let $\ph\in l(\partial U_\lm)$ and take $f\in\cF$ such that $f|_{\partial U_\lm}=\ph$.
We define $\hat f=(0\vee f)\wg1$, $\widehat{H_m f}=(0\vee H_m f)\wg1$, and $h=\widehat{H_m f}\cdot \bfone_{U_\lm}+(H_m \hat f)\cdot \bfone_{K\setminus U_\lm}$.
Since $h=H_m\hat f+(\widehat{H_m f}-H_m \hat f)\cdot \bfone_{U_\lm}$, $h$ belongs to $\cF$ from \Lem{zero}.
Moreover, since $h=\hat f$ on $V_m$, we have
\[
0\le \cE(h)-\cE(H_m \hat f)
=\frac12\mu_{\la \widehat{H_m f}\ra}(U_\lm)
-\frac12\mu_{\la H_m \hat f\ra}(U_\lm)
\]
from \Lem{en}. Therefore,
$\mu_{\la H_m\hat f\ra}(U_\lm)\le \mu_{\la \widehat{H_m f}\ra}(U_\lm)\le \mu_{\la {H_m f}\ra}(U_\lm)$.
This indicates the Markov property of $\cQ$.
\end{proof}

Let $m\in \Z_+$ and $x,y\in V_m$.
We write $x\conn{m}y$ if there exists $\lm\in\Lm_m$ such that $x,y\in \partial U_\lm$.
We say that $\gm_m=\{x_0,x_1,\dots,x_M\}$ with $x_i\in V_m$ is an $m$-walk connecting $x$ and $y$ if $x_0=x$, $x_M=y$, and $x_i\conn{m}x_{i+1}$ for every $i=0,1,\dots,M-1$.
The length $\ell_\bfh^{(m)}(\gm_m)$ of $\gm_m$ based on $\bfh$ is defined as
\[
\ell_\bfh^{(m)}(\gm_m)=\sum_{i=1}^M |\bfh(x_i)-\bfh(x_{i-1})|_{\R^N}.
\]
For $n\ge m$ and a continuous curve $\gm\in C([0,1]\to K)$ connecting $x$ and $y$, we define an $n$-walk $\pi_n(\gm)=\{x_0,x_1,\dots,x_M\}$ by
$x_0=x$ and $x_i=\gm(t_i)$ with $t_i=\inf\{t>t_{i-1}\mid \gm(t)\in V_n\setminus \{x_{i-1}\}\}$, inductively. Here, we set $t_0=0$ by convention.
It is evident that $\ell_\bfh^{(n)}(\pi_n(\gm))$ is nondecreasing in $n$. From \Lem{lh},
\begin{equation*}
\ell_\bfh(\gm)=\lim_{n\to\infty}\ell_\bfh^{(n)}(\pi_n(\gm)).
\end{equation*}
For $n\ge m$, we define
\begin{align*}
\hat \rho_\bfh^{(n)}(x,y)&=\inf\{ \ell_\bfh^{(n)}(\gm_n)\mid \gm_n\text{ is an $n$-walk connecting $x$ and $y$}\}\\
\shortintertext{and}
\hat \rho_\bfh(x,y)&=\lim_{n\to\infty}\hat \rho_\bfh^{(n)}(x,y).
\end{align*}
We remark that $\hat \rho_\bfh(x,y)=\sup_{n\ge m}\hat \rho_\bfh^{(n)}(x,y)$ since $\hat \rho_\bfh^{(n)}(x,y) $ is nondecreasing in $n$.
\begin{prop}\label{prop:rhoh}
For $x,y\in V_*$, $\rho_\bfh(x,y)=\hat \rho_\bfh(x,y)$.
In other words,
\begin{equation}\label{eq:rhoh}
 \inf_{\gm} \sup_n \ell_\bfh^{(n)}(\pi_n(\gm))
 =\sup_n \inf_{\gm_n} \ell_\bfh^{(n)}(\gm_n),
\end{equation}
where $\gm$ is taken over all the continuous curves connecting $x$ and $y$, and $\gm_n$ is taken over all the $n$-walks connecting $x$ and $y$.
\end{prop}
\begin{proof}
From the definition, the right-hand side of \Eq{rhoh} is dominated by the left-hand side.
We prove the converse inequality.

For $n\ge k\ge m$ and an $n$-walk $\gm_n=\{x_0,x_1,\dots, x_M\}$ with $x_0,x_M\in V_m$, let $\pi_{n,k}(\gm_n)$ denote the $k$-walk $\{x'_0,x'_1,\dots,x'_{M'}\}$ defined as $x'_0=x_0$ and $x'_i=x_{j(i)}$ with $j(i)=\inf\{j>j(i-1)\mid x_j\in V_k\setminus\{x'_{i-1}\}\}$ for $i=1,2,\dots$, inductively, where we set $j(0)=0$.

Let $x,y\in V_m$ for $m\in\Z_+$.
For each $n\ge m$, there exists a self-avoiding $n$-walk $\hat \gm_n$ that attains $\inf_{\gm_n} \ell_\bfh^{(n)}(\gm_n)$ on the right-hand side of \Eq{rhoh}, since there are only a finite number of self-avoiding $n$-walks.
For any divergent increasing sequence $\{n(k)\}$ and $n\ge m$, we can take a subsequence $\{n(k_j)\}$ such that $n(k_1)\ge n$ and $\{\pi_{n(k_j),n}(\hat \gm_{n(k_j)})\}_{j=1}^\infty$ are all the same.
By the diagonalization argument, we can take a divergent sequence $\{n(k)\}$ such that for every $k$ and $j$ with $k\ge j\ge m$, $\pi_{n(k),j}(\hat\gm_{n(k)})=\pi_{n(j),j}(\hat\gm_{n(j)})$.
Since $\{\pi_{n(j),j}(\hat\gm_{n(j)})\}_{j=m}^\infty$ is consistent in the sense that $\pi_{k,j}(\pi_{n(k),k}(\hat \gm_{n(k)}))=\pi_{n(j),j}(\hat\gm_{n(j)})$ for $k\ge j\ge m$, in view of \Lem{diam}, we can construct $\gm\in C([0,1]\to K)$ and a sequence of partitions $\{\Delta^{(j)}\colon 0=t^{(j)}_0<t^{(j)}_1<\dots<t^{(j)}_{N(j)}=1\}_{j=m}^\infty$ such that $\Delta^{(m)}\subset \Delta^{(m+1)}\subset\cdots$, $\lim_{j\to\infty}|\Delta^{(j)}|=0$, and $\pi_{n(j),j}(\hat\gm_{n(j)})=\{\gm(t^{(j)}_0),\gm(t^{(j)}_1),\dots,\gm(t^{(j)}_{N(j)})\}$ for all $j\ge m$. Then,
\begin{align*}
\sup_j \ell_\bfh^{(j)}(\pi_j(\gm))
&=\ell_\bfh(\gm)\\
&=\sup_j \ell_\bfh^{(j)}(\pi_{n(j),j}(\hat\gm_{n(j)}))\\
&\le \sup_j \ell_\bfh^{(n(j))}(\hat\gm_{n(j)})\\
&=\sup_j \inf_{\gm_{n(j)}} \ell_\bfh^{({n(j)})}(\gm_{n(j)})\\
&=\sup_n \inf_{\gm_n} \ell_\bfh^{(n)}(\gm_n),
\end{align*}
and equation~\Eq{rhoh} holds with $=$ replaced by $\le$.
\end{proof}
\begin{prop}\label{prop:geodesic}
Let $x,y\in K$.
There exists a shortest path $\gm$ connecting $x$ and $y$.
\end{prop}
\begin{proof}
First, we note that $\gm$ in the proof of \Prop{rhoh} is a shortest path connecting $x$ and $y$. Therefore, the claim is true for $x,y\in V_*$.

We prove the claim for $x,y\in K$ with $x\ne y$.
For $m\in \Z_+$, we define
$U_m(x)=\bigcup_{\lm\in\Lm_m,\ x\in\overline{U_\lm}}\overline{U_\lm}$, and $U_m(y)$ in the same manner.
We note that $\partial U_m(x)\subset V_m$.
There exists $m\in\Z_+$ such that $U_n(x)\cap U_n(y)=\emptyset$ for all $n\ge m$.
For $n\ge m$, take $(x^{(n)},y^{(n)})\in \partial U_n(x)\times \partial U_n(y)$ such that $\rho_\bfh(x^{(n)},y^{(n)})=\min\{\rho_\bfh(x',y')\mid (x',y')\in \partial U_n(x)\times \partial U_n(y)\}$.
Since any continuous curve $\gm$ connecting $x$ and $y$ passes some points of $\partial U_n(x)$ and $\partial U_n(y)$, respectively, $\rho_\bfh(x^{(n)},y^{(n)})$ is nondecreasing in $n$ and $\ell_\bfh(\gm)\ge\rho_\bfh(x^{(n)},y^{(n)})$.
Therefore, 
\begin{equation}\label{eq:geodesic}
\rho_\bfh(x,y)\ge\rho_\bfh(x^{(n)},y^{(n)})
\quad\mbox{for } n\ge m.
\end{equation}
If $\rho_\bfh(x^{(n)},y^{(n)})=\infty$ for some $n$, the claim is trivially true.
We assume that $\rho_\bfh(x^{(n)},y^{(n)})<\infty$ for every $n$.
For each $n\ge m$, take a shortest path $\gm_n\in C([0,1]\to K)$ connecting $x^{(n)}$ and $y^{(n)}$.
For each $n$ and $k$ with $n\ge k\ge m$, we define $s_{n,k}=\inf\{t\in[0,1]\mid \gm_n(t)\in V_k\}$, $t_{n,k}=\sup\{t\in[0,1]\mid \gm_n(t)\in V_k\}$, $x^{(n,k)}=\gm_n(s_{n,k})$, and $y^{(n,k)}=\gm_n(t_{n,k})$.
Since $V_k$ is a finite set, by the diagonalization argument, we can take $\{x_k\}_{k=m}^\infty, \{y_k\}_{k=m}^\infty\subset K$, a monotone increasing sequence $\{n(l)\}_{l=0}^\infty$ of natural numbers such that $n(0)\ge m$, $x_k,y_k\in V_k$ for all $k$, and $x^{(n(l),k)}=x_{k}$ and $y^{(n(l),k)}=y_{k}$ for all $l$ and $k$ with $l\ge k-m\ge 0$.
Define $\gm\in C([0,1]\to K)$ by connecting and reparametrizing
  $\gm_{n(l)}|_{[s_{n(l),l+m},s_{n(l),l+m-1}]}$ $(l=\dots,3,2,1)$, $\gm_{n(0)}|_{[s_{n(0),m},t_{n(0),m}]}$, and $\gm_{n(l)}|_{[t_{n(l),l+m-1},t_{n(l),l+m}]}$ $(l=1,2,3,\dots)$.
Then, $\gm$ connects $x$ and $y$ and passes all $x_n$ and $y_n$ $(n\ge m)$.
By construction, $\ell_\bfh(\gm|_{[s_n,t_n]})=\rho_\bfh(x_n,y_n)$, where $s_n$ and $t_n$ are the times such that $\gm(s_n)=x_n$ and $\gm(t_n)=y_n$.
Then, we have 
\begin{equation}\label{eq:ellh}
 \ell_\bfh(\gm)= \lim_{n\to\infty} \ell_\bfh(\gm|_{[s_n,t_n]})=\lim_{n\to\infty}\rho_\bfh(x_n,y_n).
\end{equation}
Combining this equation with \Eq{geodesic}, we obtain $\ell_\bfh(\gm)\le\rho_\bfh(x,y)$.
Therefore, $\gm$ is a shortest path connecting $x$ and $y$.
\end{proof}
We remark that identity \Eq{ellh} is true even if $\rho_\bfh(x^{(n)},y^{(n)})=\infty$ for some $n$.
\begin{cor}\label{cor:rhoh}
Let $x$ and $y$ be distinct elements of $K$.
Then, there exist sequences $\{x_n\}_{n=m}^\infty$ and $\{y_n\}_{n=m}^\infty$ of $K$ for some $m$ such that 
\begin{equation}\label{eq:xy}
x_n,y_n\in V_n\mbox{ for all $n$, }
\lim_{n\to\infty}d_K(x_n,x)=\lim_{n\to\infty}d_K(y_n,y)=0
\end{equation}
and 
\[
\rho_\bfh(x,y)=\lim_{k\to\infty}\left(\lim_{n\to\infty}\rho_\bfh(x_k,y_n)\right).
\]
\end{cor}
\begin{proof*}
Take a shortest path $\gm$ connecting $x$ and $y$, $\{x_n\},\{y_n\}\subset K$, and $\{s_n\}, \{t_n\}\subset [0,1]$ in the proof of \Prop{geodesic}.
Then, \Eq{xy} holds and
\[
\rho_\bfh(x,y)=\ell_\bfh(\gm)=\lim_{k\to\infty}\left(\lim_{n\to\infty}\ell_\bfh(\gm|_{[s_k,t_n]})\right)
=\lim_{k\to\infty}\left(\lim_{n\to\infty}\rho_\bfh(x_k,y_n)\right).\eqno{\qedsymbol}
\]
\end{proof*}
\begin{lem}\label{lem:dh}
For each $x\in K$, $\sd_\bfh(x,y)\in[0,+\infty]$ is continuous in $y\in K$.
\end{lem}
\begin{proof}
Let $x\in K$ and $M>0$.
There exists a maximal element of 
\[
\cD=\{f\in\cF\mid f(x)=0,\ f\le M,\mbox{ and }\mu_{\la f\ra}\le \mu_{\la \bfh\ra}\},
\]
in that there exists $g\in\cD$ such that $g\ge f$ $\mu$-a.e.\ for all $f\in\cD$.
Indeed, from \Lem{en}, it suffices to take $f_1,f_2,\dots$ from $\cD$ such that $\int_K f_k\,d\mu$ converges increasingly to $\sup\left\{\int_K f\,d\mu\;\vrule\;f\in\cD\right\}$ and define $g$ as $\sup_k f_k$.
Since $\cF\subset C(K)$, $g\ge f$ on $K$ for all $f\in\cD$.
By the definition of $\sd_\bfh$, $g$ is identical to $\sd_\bfh(x,\cdot)\wg M$.
This indicates the claim.
\end{proof}
Now, we prove \Thm{1}.
\begin{proof*}[Proof of \Thm{1}]
We divide the proof into two steps.

(Step 1) The case when $x,y\in V_m$ for some $m\in\Z_+$.
Take $n$ such that $n\ge m$.
We define $\ph_n(z):=\hat\rho_\bfh^{(n)}(x,z)$ for $z\in V_n$.
Then,
\begin{equation}\label{eq:lip}
|\ph_n(z)-\ph_n(z')|\le |\bfh(z)-\bfh(z')|_{\R^N}
\quad\text{for }z,z'\in V_n \text{ with }z\conn{n}z'. 
\end{equation}
Indeed, there exists an $n$-walk $\gm_n=\{x_0,x_1,\dots,x_M\}$ connecting $x$ and $z$ such that $\ph_n(z)=\ell_\bfh^{(n)}(\gm_n)$.
Since $\gm'_n:=\{x_0,x_1,\dots,x_M,z'\}$ is an $n$-walk connecting $x$ and $z'$, we have
\[
\ph_n(z')\le \ell_\bfh^{(n)}(\gm'_n)=\ph_n(z)+|\bfh(z)-\bfh(z')|_{\R^N}.
\]
By exchanging the roles of $z$ and $z'$, we obtain \Eq{lip}.

Let $\lm\in\Lm_n$ and take $\{b_{pq}\}_{p,q\in\partial U_\lm}$ in \Lem{expression}.
We denote $H_n \ph_n$ by $f_n$.
Then,
\begin{align*}
\mu_{\la f_n\ra}(U_\lm)
&=\frac12\sum_{p,q\in \partial U_\lm}b_{pq}(f_n(p)-f_n(q))^2\\
&\le\frac12\sum_{p,q\in \partial U_\lm}b_{pq}|\bfh(p)-\bfh(q)|_{\R^N}^2
\quad\text{(from \Eq{lip})}\\
&=\frac12\sum_{j=1}^N \sum_{p,q\in \partial U_\lm}b_{pq}(h_j(p)-h_j(q))^2\\
&=\sum_{j=1}^N \mu_{\la H_n h_j\ra}(U_\lm)\\
&\le \mu_{\la\bfh\ra}(U_\lm).\quad\text{(from \Lem{min})}
\end{align*}
In particular, $\mu_{\la f_n\ra}(U_\lm)\le \mu_{\la\bfh\ra}(U_\lm)$ for all $\lm\in\bigcup_{l=0}^n \Lm_l$.
Since $\sup_n \cE(f_n)\le \mu_{\la \bfh\ra}(K)/2$, the sequence $\{f_n\wg M\}_{n=m}^\infty$ is bounded in $\cF$ for any $M>0$.
There exists a subsequence $\{f_{n(k)}\wg M\}_{k=1}^\infty$ such that its Ces\`aro mean converges strongly in $\cF$. Denoting the limit by $f^M$, we have
\begin{align*}
\mu_{\la f^M\ra}(U_\lm)^{1/2}
&=\lim_{k\to\infty}\mu_{\left\la \frac1{k}\sum_{j=1}^k (f_{n(j)}\wg M)\right\ra}(U_\lm)^{1/2}\\
&\le \liminf_{k\to\infty}\frac1{k}\sum_{j=1}^k \mu_{\la f_{n(j)}\ra}(U_\lm)^{1/2}
\le \mu_{\la \bfh\ra}(U_\lm)^{1/2}
\end{align*}
for all $\lm\in\bigcup_{l=0}^\infty \Lm_l$.
Therefore, $\mu_{\la f^M\ra}\le \mu_{\la \bfh\ra}$ by the monotone class theorem.
Since the convergence in $\cF$ indicates uniform convergence from (A2),
$f^M(x)=0$ and $f^M(y)=\rho_\bfh(x,y)\wg M$ from \Prop{rhoh}.
Thus,
\[
  \sd_\bfh(x,y)\ge f^M(y)-f^M(x)=\rho_\bfh(x,y)\wg M.
\]
Since $M$ is arbitrary, we obtain $\sd_\bfh(x,y)\ge \rho_\bfh(x,y)$.

(Step 2) The case when $x,y\in K$. We may assume that $x\ne y$.
Take $\{x_n\},\{y_n\}\subset K$ in \Cor{rhoh}. Then, from \Lem{dh}, Step~1, and \Cor{rhoh},
\[
\sd_\bfh(x,y)=\lim_{k\to\infty}\left(\lim_{n\to\infty}\sd_\bfh(x_k,y_n)\right)\ge \lim_{k\to\infty}\left(\lim_{n\to\infty}\rho_\bfh(x_k,y_n)\right)=\rho_\bfh(x,y).\eqno{\qedsymbol}
\]
\end{proof*}
The following is a remark on the topologies of $K$ induced by $\rho_\bfh$ and $\sd_\bfh$.
\begin{cor}\label{cor:metric}
Suppose that $\rho_\bfh$ is a $[0,+\infty]$-valued metric on $K$.
Then, $\sd_\bfh$ is also a $[0,+\infty]$-valued metric.
Moreover, both $\rho_\bfh$ and $\sd_\bfh$ provide the same topologies on $K$ as the original one.
\end{cor}
\begin{proof}
From \Thm{1}, the first claim follows and the topology $\cO_\bfh$ associated with $\sd_\bfh$ is stronger than that with $\rho_\bfh$.
From \Lem{dh}, $\cO_\bfh$ is weaker than the original topology on $K$.
Since a continuous bijective map from a compact Hausdorff space to a Hausdorff space is homeomorphic,
by applying this fact to the identity map from $(K,d_K)$ to $(K,\rho_\bfh)$, the second assertion holds.
\end{proof}
\section{Proof of \Thm{2}}
Throughout this section, we assume (B1)--(B4).
Furthermore, we follow the notation used in Section~2.

For $w\in W_m$ with $m\in\Z_+$, $V_w$ denotes $K_w\cap V_m$.
For $w=w_1w_2\cdots w_m\in W_m$ and $w'=w'_1w'_2\cdots w'_{m'}\in W_{m'}$, $w_1w_2\cdots w_mw'_1w'_2\cdots w'_{m'}\in W_{m+m'}$ is represented as $ww'$.
For $i\in S$, $i^n\in W_n$ and $i^\infty\in \Sg$ denote $\underbrace{ii\cdots i}_n$ and $iii\cdots$, respectively.

The Dirichlet forms associated with regular harmonic structures have a property stronger than (A2): there exists $c>0$ such that 
\begin{equation}\label{eq:poincare}
\biggl(\sup_{y\in K}f(y)-\inf_{x\in K}f(x)\biggr)^2\le c\cE(f),
\quad f\in\cF.
\end{equation}
In particular, by using \Thm{1}, $\rho_\bfh(x,y)\le \sd_\bfh(x,y)\le \sqrt{c\mu_{\la\bfh\ra}(K)/2}<+\infty$ for any $x,y\in K$.

Let $p\in V_0$ and take $i\in S_0$ such that $\psi_i(p)=p$.
Recall that $v_i$ is an eigenvector of $A_i$ whose components are all nonnegative.
Let $u_i$ be the column vector $(D_{p'p})_{p'\in V_0}$.
Then, $u_i$ is an eigenvector of ${}^t\! A_i$ with respect to the eigenvalue $r_i$ (\cite[Lemma~5]{HN06}).
Since $K\setminus V_0$ is connected by (B2), (B3), and \cite[Proposition~1.6.8]{Ki}, from \cite[Theorem~3.2.11]{Ki},
\begin{equation}\label{eq:indep}
u_i(q)>0 \mbox{ for all }q\in V_0\setminus\{ p\}.
\end{equation}
We normalize $v_i$ so that $(u_i,v_i)_{l(V_0)}=1$.
The element of $l(V_0)$ taking constant~$1$ will be denoted by $\bfone$.
Let $\tilde l(V_0)=\{u\in l(V_0)\mid (u,\bfone)_{l(V_0)}=0\}$ and let $P\colon l(V_0)\to l(V_0)$ be the orthogonal projection onto $\tilde l(V_0)$.
We note that $u_i\in \tilde l(V_0)$ by $D\bfone=0$ and the definition of $u_i$.
\begin{lem}[{cf.~\cite[Lemma~6]{HN06}}]\label{lem:convergence}
Let $u\in l(V_0)$. Then, 
\[
\lim_{n\to\infty} r_i^{-n}P A_i^n u=(u_i,u)_{l(V_0)}P v_i.
\]
In particular, for $q_1,q_2\in V_0$,
\[
\lim_{n\to\infty} r_i^{-n}\left(A_i^n u(q_1)-A_i^n u(q_2)\right)=(u_i,u)_{l(V_0)}(v_i(q_1)-v_i(q_2)).
\]
Both convergences are uniform on the set $\{u\in l(V_0)\mid |P u|_{l(V_0)}\le 1\}$.
\end{lem}

We recall a property of energy measures as follows.
\begin{lem}[{cf.~\cite[Lemma~3.11]{Hi05}}]\label{lem:energy}
For $f\in \cF$ and $m\in\Z_+$, we have
  \[
    \mu_{\la f\ra}
    =\sum_{w\in W_m}\frac1{r_w} 
    (\psi_w)_* \mu_{\la \psi_w^* f\ra},
  \]
  that is, $\mu_{\la f\ra}(A)=\sum_{w\in W_m}\frac1{r_w} 
    \mu_{\la \psi_w^* f\ra}(\psi_w^{-1}(A))$ for any Borel subset $A$ of $K$.
\end{lem}
The following is a rough upper-side estimate of $\sd_\bfh$ by $\rho_\bfh$. 
\begin{lem}\label{lem:2.1}
Let $m\in\Z_+$, $w\in W_m$, and $x,y\in V_w$ with $x\ne y$.
Let $\gm\in C([0,1]\to K)$ be a shortest path connecting $x$ and $y$, and suppose that the image of $\gm$ is contained in $K_w$.
For each $n\in\N$, we define $z_n\in V_{m+n}$ by
\[
z_n=\gm(s_n)\mbox{ with }
s_n=\inf\{t\in(0,1]\mid \gm(t)\in V_{m+n}\setminus\{x\}\}.
\]
Then, there exists $c_0(n)>0$ for each $n\in\N$ that is independent of $m,w,x,y,\gm$ such that $\rho_\bfh(x,z_n)\ge c_0(n)\sd_\bfh(x,y)$.
\end{lem}
\begin{proof}
Let $p\in V_0$ and take $i\in S_0$ such that $\psi_i(p)=p$.
Let $q$ and $q'$ denote the distinct elements of $V_0\setminus\{p\}$, that is, $V_0=\{p,q,q'\}$.
Define $\a\in \tilde l(V_0)$ by $\a(p)=1$, $\a(q)=-1$, and $\a(q')=0$.
From \Eq{indep}, $\a$ and $u_i$ are linearly independent in $l(V_0)$; thus, the linear span of $\a$ and $u_i$ is $\tilde l(V_0)$ since $\dim \tilde l(V_0)=2$ from (B1).
Therefore, there exists $\dl>0$ such that any $u\in l(V_0)$ with $|P u|_{l(V_0)}=1$ satisfies $|(u,\a)_{l(V_0)}|\ge\dl$ or $|(u,u_i)_{l(V_0)}|\ge\dl$.
Let $\hat q$ denote $q$ or $q'$.
From \Lem{convergence}, 
\begin{align*}
A_i^n u(\hat q)-A_i^n u(p)
=r_i^{n}(u_i,u)_{l(V_0)}v_i(\hat q)+o(r_i^{n})
\quad \mbox{as }n\to\infty
\end{align*}
uniformly on $\{u\in l(V_0)\mid |P u|_{l(V_0)}= 1\}$.
Therefore, for sufficiently large $M\in \N$, 
\[
|A_i^M u(\hat q)-A_i^M u(p)|_{l(V_0)}\ge \frac{r_i^M\dl}2 v_i(\hat q)\,(>0)
\]
for any $u\in l(V_0)$ such that $|P u|_{l(V_0)}= 1\mbox{ and }|(u,u_i)_{l(V_0)}|\ge\dl$.

From this argument, the map
\[
\tilde l(V_0)\ni u\mapsto \left((u,\a)_{l(V_0)}^2+|A_i^M u(\hat q)-A_i^M u(p)|_{l(V_0)}^2\right)^{1/2}\in\R
\]
defines a norm on $\tilde l(V_0)$;
so do the maps $u\mapsto\cE(\iota(u))^{1/2}$ and $u\mapsto\cE(\psi_{i^n}^*(\iota(u)))^{1/2}$ for $n\in\N$ because of (B4).
Then, there exist $c_1>0$ and $c_2>0$ such that for every $h\in\cH$, 
\begin{align*}
\cE(h)&\le c_1\cE(\psi_{i^n}^*(h))\\
\shortintertext{and}
\cE(h)&\le c_2\left((h(q)-h(p))^2+((\psi_{i^M}^* h)(\hat q)-(\psi_{i^M}^* h)(p))^2\right)
\end{align*}
for all $q, \hat q\in V_0\setminus\{p\}$.
Since there are only finitely many choices of $p$, $q$, and $\hat q$, we can take $c_1$ and $c_2$ as constants independent of $p$, $q$, and $\hat q$.
(Note that $c_1$ depends on $n$.)

Now, in the setting of the claim, let $f\in\cF$ satisfy $\mu_{\la f\ra}\le\mu_{\la \bfh\ra}$.
From \Lem{energy}, $\mu_{\la \psi_w^*f\ra}\le\mu_{\la \psi_w^*\bfh\ra}$.
In particular, $\cE(\psi_w^*f)\le \sum_{j=1}^N \cE(\psi_w^* h_j)$.
Let $p=\psi_w^{-1}(x)\in V_0$ and take $i\in S_0$ such that $\psi_i(p)=p$.
Take $k,l\in S_0\setminus\{i\}$ such that $z_n=\pi(w i^n k^\infty)$ and $z_{n+M}=\pi(w i^{n+M} l^\infty)$, and set $q=\pi(k^\infty)$, $\hat q=\pi(l^\infty)$.
Then,
\begin{align*}
(f(y)-f(x))^2&=((\psi_w^* f)(\psi_w^{-1}(y))-(\psi_w^* f)(p))^2\\
&\le c\cE(\psi_w^* f)\quad\mbox{(from \Eq{poincare})}\\
&\le c \sum_{j=1}^N \cE(\psi_w^* h_j)\\
&\le cc_1 \sum_{j=1}^N \cE(\psi_{w i^n}^* h_j)\\
&\le cc_1c_2\sum_{j=1}^N \Bigl\{((\psi_{w i^n}^* h_j)(q)-(\psi_{w i^n}^* h_j)(p))^2\\*
&\hspace{5em}+((\psi_{i^M}^*\psi_{w i^n}^* h_j)(\hat q)-(\psi_{i^M}^*\psi_{w i^n}^* h_j)(p))^2\Bigr\}\\
&= cc_1c_2\sum_{j=1}^N \left\{(h_j(z_n)-h_j(x))^2+(h_j(z_{n+M})-h_j(x))^2\right\}\\
&\le 2cc_1c_2\rho_\bfh(x,z_n)^2.
\end{align*}
Thus, $\sd_\bfh(x,y)\le(2cc_1c_2)^{1/2}\rho_\bfh(x,z_n)$.
\end{proof}
\begin{cor}\label{cor:2.2}
Following the same notation as that in \Lem{2.1}, we have
\begin{enumerate}
\item $\rho_\bfh(x,y)\ge c_0(n)\sd_\bfh(x,y)$;
\item $\rho_\bfh(x,z_n)\ge c_0(n)\rho_\bfh(x,y)$.
\end{enumerate}
\end{cor}
\begin{proof}
From \Lem{2.1}, (i) is evident since $\rho_\bfh(x,y)\ge\rho_\bfh(x,z_n)$.
(ii) follows from \Lem{2.1} and \Thm{1}.
\end{proof}
The following technical lemma is used in the proof of \Lem{2.3}.
\begin{lem}\label{lem:neg}
Let $p\in V_0$ and $i\in S_0$ satisfy $\psi_i(p)=p$.
Let $\a_1,\dots,\a_N\in l(V_0)$ and $q\in V_0\setminus\{p\}$.
For $n\in\Z_+$, let $\gm^{(n)}_j=\left(A_i^n \a_j(q)-A_i^n \a_j(p)\right)/v_i(q)$ for $j=1,\dots,N$ and $\gm^{(n)}=\left(\sum_{j=1}^N (\gm^{(n)}_j)^2\right)^{1/2}$.
We write $\ph^{(n)}_j=\iota(A_i^n\a_j)-\gm^{(n)}_j\iota(v_i)\in\cH$.
Then, given $\dl>0$ and $\eps>0$, there exists $n_0\in\N$ that is independent of $\a_1,\cdots,\a_N, p,q$ such that for all $n\ge n_0$,
\begin{equation}\label{eq:neg}
  \cE(\ph^{(n)}_j)\le \eps (\gm^{(n)})^2 \cE(\iota(v_i)),
  \quad j=1,\dots,N,
\end{equation}
as long as $|(u_i,\a_l)_{l(V_0)}|\ge \dl \left(\sum_{j=1}^N |P\a_j|_{l(V_0)}^2\right)^{1/2}$ for some $l\in\{1,\dots,N\}$.
\end{lem}
\begin{proof}
By multiplying a constant if necessary, we may assume the additional constraint $\sum_{j=1}^N|P\a_j|_{l(V_0)}^2=1$ to prove \Eq{neg} without loss of generality.
From \Lem{convergence}, for $j=1,\dots,N$,
\begin{align}\label{eq:conv1}
\lim_{n\to\infty} r_i^{-n} PA_i^n\a_j&=(u_i,\a_j)_{l(V_0)}P v_i
\shortintertext{and}
\label{eq:conv2}
\lim_{n\to\infty} r_i^{-n} \gm^{(n)}_j&=(u_i,\a_j)_{l(V_0)}
\end{align}
uniformly on $\Gm:=\{(\a_1,\dots,\a_N)\in (l(V_0))^N\mid\sum_{j=1}^N|P\a_j|_{l(V_0)}^2=1\}$.
Therefore, 
\begin{equation}\label{eq:c1}
\lim_{n\to \infty}r_i^{-2n}\cE(\ph^{(n)}_j)=0 \mbox{ uniformly on }\Gm.
\end{equation}
By the assumption $|(u_i,\a_l)_{l(V_0)}|\ge \dl$ and \Eq{conv2}, 
\begin{equation}\label{eq:c2}
r_i^{-2n}(\gm^{(n)})^2\ge r_i^{-2n}(\gm^{(n)}_l)^2\ge\dl^2/2
\mbox{ for sufficiently large $n$}.
\end{equation}
Therefore, the assertion follows from \Eq{c1}, \Eq{c2}, and $\inf_{i\in S_0}\cE(\iota(v_i))>0$.
\end{proof}
The following is a key lemma for the proof of \Thm{2}.
\begin{lem}\label{lem:2.3}
Let $m\in\Z_+$, $w\in W_m$, and $x\in V_w$.
Take $i\in S_0$ such that $x=\pi(w i^\infty)$.
Let $\dl>0$ and $\eps>0$.
Suppose that there exists $l\in\{1,\dots,N\}$ such that 
\[
|(u_i,\a_{w,l})_{l(V_0)}|\ge \dl\biggl(\sum_{j=1}^N|P\a_{w,j}|_{l(V_0)}^2\biggr)^{1/2},
\quad\mbox{where }\a_{w,j}=\iota^{-1}(\psi_w^* h_j).
\]
Then, there exists $M\in\N$ that is independent of $m,w,x,l$ such that for $n\ge M$,
\begin{equation}\label{eq:d}
  \sd_\bfh(x,y)\le (1+\eps)|\bfh(y)-\bfh(x)|_{\R^N}
  \le (1+\eps)\rho_\bfh(x,y)
\end{equation}
for any $y\in V_{w i^n}\setminus\{x\}$.
\end{lem}
\begin{proof}
For $s\in S_0$, let $p_s$ denote the fixed point of $\psi_s$, that is, $p_s=\pi(s^\infty)$.
Let
\[
C=\max_{s\in S_0}\left\{\frac{\max_{q\in V_0\setminus\{p_s\}}v_s(q)\times(-Dv_s)(q)}{\min_{q\in V_0\setminus\{p_s\}}v_s(q)\times(-Dv_s)(q)}\right\},
\]
which is positive by (B3).
Take $\eps_1>0$ and $\eps_2>0$ such that 
\[
(1+C)(1+\eps_2)^{1/2}-C\le 1+\eps\mbox{ and }\eps_1=\eps_2/2.
\]
We remark that any $y\in V_{w i^n}\setminus\{x\}$ for $n\in\Z_+$ is described as $y=\pi(w i^n k^\infty)$ for some $k\in S_0\setminus\{i\}$.

Fix $k\in S_0\setminus\{i\}$.
For $n\in\N$ and $j=1,\dots,N$, let $\hat h_j^{(n)}$ denote $\psi_{w i^n}^* h_j$.
Note that $\hat h_j^{(n)}$ is also described as $\iota(A_i^n \a_{w,j})$.
Let 
\[
g_j^{(n)}=\left((\hat h_j^{(n)}(p_k)-\hat h_j^{(n)}(p_i))/{v_i(p_k)}\right)\iota(v_i)
\mbox{ and }
\ph_j^{(n)}=\hat h_j^{(n)}-g_j^{(n)}
\]
for $j=1,\dots,N$, and 
\[
g^{(n)}=\left(\sum_{j=1}^N\left(\frac{\hat h_j^{(n)}(p_k)-\hat h_j^{(n)}(p_i)}{v_i(p_k)}\right)^2\right)^{1/2}\iota(v_i).
\] 
We note that 
\begin{equation}\label{eq:g}
\sum_{j=1}^N \mu_{\la g_j^{(n)}\ra}=\mu_{\la g^{(n)}\ra}
\end{equation}
and
\begin{equation}\label{eq:trivial}
\mu_{\la \hat h_j^{(n)}\ra}\le (1+\eps_1)\mu_{\la g_j^{(n)}\ra}+(1+\eps_1^{-1})\mu_{\la \ph_j^{(n)}\ra},\quad j=1,\dots,N.
\end{equation}
From \Lem{neg} with $\eps=\eps_1/((1+\eps_1^{-1})N)$, there exists $M\in \N$ that is independent of $m,w,x,l,k$ such that for all $n\ge M$,
\begin{equation}\label{eq:small}
\cE(\ph_j^{(n)})\le\frac{\eps_1}{(1+\eps_1^{-1})N}\cE(g^{(n)}),\quad j=1,\dots,N.
\end{equation}
Hereafter, we fix such $n$ and omit the superscript~${}^{(n)}$ from the notation.
From \Eq{g} and \Eq{small}, we have
\begin{equation}\label{eq:dom}
\sum_{j=1}^N\left\{(1+\eps_1)\mu_{\la g_j\ra}(K)+(1+\eps_1^{-1})\mu_{\la \ph_j\ra}(K)\right\}
\le (1+\eps_2)\mu_{\la g\ra}(K).
\end{equation}
Let $f\in\cF$ satisfy $\mu_{\la f\ra}\le\mu_{\la \bfh\ra}$ and $f(x)=0$.
Let $\hat f$ denote $\psi_{w i^n}^* f$ and define $\check f:=\hat f\vee g\in\cF$.
Then,
\[
  \mu_{\la \hat f\ra}\le \sum_{j=1}^N\mu_{\la\hat h_j\ra}
  \le\sum_{j=1}^N \{(1+\eps_1)\mu_{\la g_j\ra}+(1+\eps_1^{-1})\mu_{\la \ph_j\ra}\}
\]
and
\[
\frac{d\mu_{\la \check f\ra}}{d\nu}\le \frac{d\mu_{\la \hat f\ra}}{d\nu}\vee \frac{d\mu_{\la g\ra}}{d\nu}\quad
\nu\mbox{-a.e. with }\nu=\mu_{\la \hat f\ra}+\mu_{\la g\ra}
\]
in view of \Lem{energy}, \Eq{trivial}, and \Lem{en}~(ii).
Combining these inequalities with \Eq{g}, we have
\[
\mu_{\la \check f\ra}\le\sum_{j=1}^N \{(1+\eps_1)\mu_{\la g_j\ra}+(1+\eps_1^{-1})\mu_{\la \ph_j\ra}\}.
\]
In particular,
$\mu_{\la \check f\ra}(K)\le (1+\eps_2)\mu_{\la g\ra}(K)$ from \Eq{dom}.

Let $F=(1+\eps_2)^{-1/2}\iota(\check f|_{V_0})\,(=(1+\eps_2)^{-1/2}H_0\check f)\in\cH$.
Then,
$\cE(F)\le \cE((1+\eps_2)^{-1/2}\check f)\le \cE(g)$, which implies that
\[
0\le\cE(F-g)
=\cE(F)-2\cE(F,g)+\cE(g)
\le 2\cE(g-F,g).
\]
Letting $G=g-F\in\cH$, we have $G(p_i)=0$ and
\[
  G(q)=g(q)-(1+\eps_2)^{-1/2}(\hat f(q)\vee g(q))\le  (1-(1+\eps_2)^{-1/2})g(q)
\]
for any $q\in V_0$.
Let $q'$ denote the unique element of $V_0\setminus\{p_i,p_k\}$.
Since 
\[
(G|_{V_0},-Dv_i)_{l(V_0)}=\left(\sum_{j=1}^N\left(\frac{\hat h_j(p_k)-\hat h_j(p_i)}{v_i(p_k)}\right)^2\right)^{-1/2}\cE(g-F,g)\ge0,
\]
we have
\begin{align*}
G(p_k)(-Dv_i)(p_k)
&\ge -G(p_i)(-Dv_i)(p_i)-G(q')(-Dv_i)(q')\\
&\ge 0-(1-(1+\eps_2)^{-1/2})g(q')(-Dv_i)(q')\\
&\ge -C(1-(1+\eps_2)^{-1/2})g(p_k)(-Dv_i)(p_k).
\end{align*}
Thus, 
\begin{align*}
-C(1-(1+\eps_2)^{-1/2})g(p_k)
&\le G(p_k)\\
&= g(p_k)-(1+\eps_2)^{-1/2}(\hat f(p_k)\vee g(p_k))\\
&\le g(p_k)-(1+\eps_2)^{-1/2}\hat f(p_k),
\end{align*}
which implies that
\[
\hat f(p_k)
\le((1+C)(1+\eps_2)^{1/2}-C)g(p_k)\le(1+\eps) g(p_k).
\]
Therefore, for $y=\pi(w i^n k^\infty)\in V_{w i^n}\setminus\{x\}$,
\begin{align*}
f(y)-f(x)&=f(y)
\le (1+\eps)\left(\sum_{j=1}^N\left(\frac{\hat h_j(p_k)-\hat h_j(p_i)}{v_i(p_k)}\right)^2\right)^{1/2}v_i(p_k)\\
&=(1+\eps)|\bfh(y)-\bfh(x)|_{\R^N}
\le (1+\eps)\rho_\bfh(x,y).
\end{align*}
By taking the supremum with respect to $f$, we obtain \Eq{d}.
\end{proof}
\begin{lem}\label{lem:good}
There exists $\dl'>0$ such that the following holds:
for any distinct points $i,j$ of $S_0$ and every $u\in l(V_0)$,
$|(u_i,u)_{l(V_0)}|\vee|(u_j,u)_{l(V_0)}|\ge \dl'|P u|_{l(V_0)}$.
\end{lem}
\begin{proof}
Since the linear span of $u_i$ and $u_j$ is $\tilde l(V_0)$, 
\[
\inf\{|(u_i,u)_{l(V_0)}|\vee|(u_j,u)_{l(V_0)}|\mid u\in \tilde l(V_0),\ |u|_{l(V_0)}=1\}>0.
\]
Therefore, the assertion follows.
\end{proof}
Now, we prove \Thm{2}. For the proof, we make a slight generalization of the concept of $\ell_\bfh$.
Let $\cI$ be a disjoint union of a finite number of closed intervals $\{I_k\}$. For $\gm\in C(\cI\to K)$, we define its length $\ell_\bfh(\gm)$ by $\sum_{k}\ell_\bfh(\gm|_{I_k})$.
\begin{proof*}[Proof of \Thm{2}]
From \Thm{1}, it suffices to prove the inequality $\rho_\bfh(x,y)\ge \sd_\bfh(x,y)$ for distinct $x,y\in K$.

(Step 1) The case when $x,y\in V_m$ for some $m\in\Z_+$.
Take $\dl'$ in \Lem{good}.
Let $\eps>0$, $\dl=\dl'/\sqrt N$ and take $M$ in \Lem{2.3}.
Take a shortest path $\gm\in C([0,1]\to K)$ connecting $x$ and $y$.
We may assume that $\gm$ is injective.
Let $\cI_1=[0,1]$.
We define $\{I_{n,k}\}_{k=1}^{l(n)}$, $\{J_{n,k}\}_{k=1}^{l(n)}$, and $\cI_{n+1}$ for $n\in\N$ inductively as follows.
First, let $\{I_{n,k}\}_{k=1}^{l(n)}$ be the collection of closed intervals $I_{n,k}=[s_{n,k},t_{n,k}]$ such that 
\begin{itemize}
\item $\bigcup_{k=1}^{l(n)} I_{n,k}=\cI_n$;
\item $s_{n,k}<t_{n,k}$, $\gm(s_{n,k})\in V_{m+Mn}$, $\gm(t_{n,k})\in V_{m+Mn}$, and $\gm(t)\notin V_{m+Mn}$ for all $t\in(s_{n,k},t_{n,k})$;
\item For $k\ne k'$, $I_{n,k}\cap I_{n,k'}$ consists of at most one point.
\end{itemize}
Next, for each $k=1,\dots,l(n)$, take $w\in W_{m+Mn}$ and $i,\hat i\in S_0$ such that $\gm([s_{n,k},t_{n,k}])\subset K_w$, $\gm(s_{n,k})=\pi(w i^\infty)$, and $\gm(t_{n,k})=\pi(w \hat i^\infty)$.
Denote $\iota^{-1}(\psi_w^* h_j)$ by $\a_j$ for $j=1,\dots,N$.
Take $j\in\{1,\dots,N\}$ such that $|P \a_j|_{l(V_0)}$ attains the maximum of\break $\{|P \a_1|_{l(V_0)},\dots,|P \a_N|_{l(V_0)}\}$.
From \Lem{good}, at least one of the following holds:
\begin{enumerate}
\item $|(u_{i},\a_j)_{l(V_0)}|\ge \dl'|P \a_j|_{l(V_0)}$;
\item $|(u_{\hat i},\a_j)_{l(V_0)}|\ge \dl'|P \a_j|_{l(V_0)}$.
\end{enumerate}
If (i) holds, set $J_{n,k}=[s_{n,k},t'_{n,k}]$ with $t'_{n,k}=\inf\{t>s_{n,k}\mid \gm(t)\in V_{m+M(n+1)}\}$. 
Otherwise, set $J_{n,k}=[s'_{n,k},t_{n,k}]$ with $s'_{n,k}=\sup\{t<t_{n,k}\mid \gm(t)\in V_{m+M(n+1)}\}$.
Define $\cI_{n+1}=\bigcup_{k=1}^{l(n)}\overline{I_{n,k}\setminus J_{n,k}}$.

Let $n\in\N$ and $k=1,\dots,l(n)$. From \Cor{2.2}~(ii),
\[
\ell_\bfh(\gm|_{J_{n,k}})\ge c_0(M)\ell_\bfh(\gm|_{I_{n,k}}),
\]
that is,
\[
\ell_\bfh(\gm|_{\overline{I_{n,k}\setminus J_{n,k}}})\le (1-c_0(M))\ell_\bfh(\gm|_{I_{n,k}}).
\] 
Therefore,
\[
\ell_\bfh(\gm|_{\cI_{n+1}})\le (1-c_0(M))\ell_\bfh(\gm|_{\cI_n}).
\] 
Then,
\begin{equation}\label{eq:inductive}
\ell_\bfh(\gm|_{\cI_{n}})\le (1-c_0(M))^{n-1} \ell_\bfh(\gm)
=(1-c_0(M))^{n-1} \rho_\bfh(x,y).
\end{equation}

Fix $R\in\N$ and let $\sJ=\{J_{n,k}\mid 1\le n\le R,\ 1\le k\le l(n)\}$.
Let $0=t_0<t_1<\dots<t_l=1$ be the arrangement of all the endpoints of the intervals $J_{n,k}$ in $\sJ$ in increasing order.
For all $i=0,\dots,l-1$, the inequality
\begin{equation}\label{eq:c0M}
 \rho_\bfh(\gm(t_i),\gm(t_{i+1}))\ge c_0(M)\sd_\bfh(\gm(t_i),\gm(t_{i+1}))
\end{equation}
holds by applying \Cor{2.2}~(i) to a series of adjacent two points of a suitable $n$-walk connecting $\gm(t_i)$ and $\gm(t_{i+1})$, where $n$ is the smallest number such that $\gm(t_i)\in V_n$ and $\gm(t_{i+1})\in V_n$.
Let $\cQ=\{i=0,\dots,l-1\mid [t_i,t_{i+1}]\in\sJ\}$.
From \Lem{2.3}, for $i\in \cQ$,
\[
 \rho_\bfh(\gm(t_i),\gm(t_{i+1}))\ge (1+\eps)^{-1}\sd_\bfh(\gm(t_i),\gm(t_{i+1})).
\]
Then,
\begin{align*}
\rho_\bfh(x,y)
&=\sum_{i=0}^{l-1}\rho_\bfh(\gm(t_i),\gm(t_{i+1}))
\ge\sum_{i\in \cQ}\rho_\bfh(\gm(t_i),\gm(t_{i+1}))\\
&\ge (1+\eps)^{-1}\sum_{i\in \cQ}\sd_\bfh(\gm(t_i),\gm(t_{i+1}))\\
&\ge(1+\eps)^{-1}\sum_{i=0}^{l-1}\sd_\bfh(\gm(t_i),\gm(t_{i+1}))\\
&\quad-(1+\eps)^{-1}c_0(M)^{-1}\sum_{i\notin \cQ}\rho_\bfh(\gm(t_i),\gm(t_{i+1}))
\quad\mbox{(from \Eq{c0M})}\\
&\ge(1+\eps)^{-1}\sd_\bfh(x,y)
-(1+\eps)^{-1}c_0(M)^{-1}\ell_\bfh(\gm|_{\cI_{R+1}})\\
&\ge(1+\eps)^{-1}\sd_\bfh(x,y)-(1+\eps)^{-1}c_0(M)^{-1}(1-c_0(M))^{R}\rho_\bfh(x,y).
\end{align*}
Here, \Eq{inductive} was used in the last inequality.
By letting $R\to\infty$ and $\eps\to0$, we conclude that $\rho_\bfh(x,y)\ge\sd_\bfh(x,y)$.

(Step 2) The case when $x,y\in K$.
Take $\{x_n\},\{y_n\}\subset K$ in \Cor{rhoh}. Then, from \Cor{rhoh}, Step~1, and \Lem{dh},
\[
\rho_\bfh(x,y)=\lim_{k\to\infty}\left(\lim_{n\to\infty}\rho_\bfh(x_k,y_n)\right)\ge \lim_{k\to\infty}\left(\lim_{n\to\infty}\sd_\bfh(x_k,y_n)\right)=\sd_\bfh(x,y).\eqno{\qedsymbol}
\]
\end{proof*}
\subsection*{Acknowledgements}
This research was partly supported by KAKENHI (21740094, 24540170).
The author thanks Naotaka Kajino for insightful discussions and the anonymous referee for very careful reading and valuable proposals which have led to improvements of the first version.

\end{document}